\documentclass[2.11,usenames,dvipsnames,svgnames]{article}
\usepackage[left=1.25in, right=1.25in, top=1in, bottom=1in]{geometry}
\usepackage{amsmath,amsfonts,amsthm}
\numberwithin{equation}{subsection}
\usepackage{algorithm}
\usepackage{algorithmicx}
\usepackage{algpseudocode}
\usepackage{amssymb}
\usepackage{graphicx,color,epstopdf}
\usepackage{cite}
\usepackage{mathrsfs,mathtools}
\usepackage{xcolor}
\usepackage{enumitem}
\usepackage{hyperref}

\usepackage{chngcntr}
\counterwithout{equation}{subsection}
\counterwithin{equation}{section}

\newtheorem{theorem}{Theorem}
\newtheorem{lemma}[theorem]{Lemma}
\newtheorem{corollary}[theorem]{Corollary}
\newtheorem{remark}[theorem]{Remark}

\usepackage{bbm,mathabx}

\usepackage{array,multirow,booktabs} 

\makeatletter
\def\BState{\State\hskip-\ALG@thistlm}
\makeatother

\newcommand{\bs}[1]{\boldsymbol{#1}}
\newcommand{\bmu}{\bs{\mu}}
\newcommand{\bPhi}{\bs{\Phi}}
\newcommand{\brho}{\bs{\rho}}
\newcommand{\bx}{\bs{x}}
\newcommand{\bGamma}{\bs{\Gamma}}
\newcommand{\dx}[1]{\mathrm{d} #1}
\newcommand{\bbm}[1]{\mathbbm{#1}}
\newcommand{\N}{\mathbbm{N}}
\newcommand{\R}{\mathbbm{R}}
\newcommand{\E}{\mathbbm{E}}
\newcommand{\cF}{\mathcal{F}}
\newcommand{\cN}{\mathcal{N}}
\newcommand{\cS}{\mathcal{S}}
\newcommand{\truth}{^{\mathcal{N}}}
\newcommand{\bnu}{\bs{\nu}}

\DeclareMathOperator*{\argmax}{argmax}

\newcommand{\dul}[1]{\underline{\underline{#1}}}

\newcommand{\naive}{na\"{\i}ve}

\newcommand{\rev}[1]{{\leavevmode\color{black}{#1}}}

\newcommand\yc[1]{\textcolor{black}{#1}}

\newcommand\yanlai[1]{\textcolor{black}{#1}}

\ifpdf
  \DeclareGraphicsExtensions{.eps,.pdf,.png,.jpg}
\else
  \DeclareGraphicsExtensions{.eps}
\fi

\newcommand{\TheTitle}{A goal-oriented RBM-Accelerated generalized polynomial chaos algorithm}
\newcommand{\TheAuthors}{J. Jiang, Y. Chen, and A. Narayan}

\newcommand{\headers}[2]{
\pagestyle{myheadings}
\markboth{\uppercase\expandafter{#2}}
{}}

\newcommand{\email}[1]{\protect\href{mailto:#1}{#1}}
\newcommand\funding[1]{\protect\\ \hspace*{1.8em}{\color{black}\bfseries Funding:} #1}

\headers{\TheTitle}{\TheAuthors}

\title{{\TheTitle}\thanks{Submitted to the editors on 9/12/2016.
\funding{This work was funded by the National Science Foundation (NSF) under awards DMS-1216928 and DMS-1552238, the Air Force Office of Scientific Research (AFOSR) under award FA9550-15-1-0467, and the Defense Advanced Research Projects Agency (DAPRA) under award N660011524053.}}}

\author{
  Jiahua Jiang\thanks{Mathematics Department, University of Massachusetts Dartmouth, North Dartmouth, MA (\email{jjiang@umassd.edu}, \email{yanlai.chen@umassd.edu}).}
  \and
  Yanlai Chen\footnotemark[2]
  \and
  Akil Narayan\thanks{Department of Mathematics, and Scientific Computing and Imaging Institute, University of Utah, Salt Lake City, UT (\email{akil@sci.utah.edu}).}
}

\usepackage{amsopn}

\begin{document}

\maketitle


\begin{abstract}

The non-intrusive generalized Polynomial Chaos (gPC) method is a popular computational approach for 
solving partial differential equations (PDEs) with random inputs. The main hurdle preventing its efficient direct application 
for high-dimensional input parameters is that the size of many parametric sampling meshes grows exponentially in the number of inputs (the ``curse of dimensionality'').  In this paper, we design a weighted version of the reduced basis method (RBM) for use in the non-intrusive gPC framework. \rev{We construct an RBM surrogate that can rigorously achieve a user-prescribed error tolerance, and ultimately is used to more efficiently compute a gPC approximation non-intrusively.} The algorithm is capable of speeding up traditional non-intrusive gPC methods by orders of magnitude without degrading accuracy, \rev{assuming that the solution manifold has low Kolmogorov width}. \rev{Numerical experiments on our test problems show that the relative efficiency improves as the parametric dimension increases}, demonstrating the potential of the method in delaying the curse of dimensionality. Theoretical results as well as numerical evidence justify these findings.
\end{abstract}



\section{Introduction}
Computational methods for stochastic problems in uncertainty quantification (UQ) are an increasingly-important area of research and much recent effort in this direction has been rewarded with many promising developments. In particular, algorithms that quantify the effect of (potentially) random input parameters on solutions to differential equations have seen rapid advancement. One of the most widely used methods in this context is the generalized Polynomial Chaos (gPC) method \cite{xiu_wiener--askey_2002}, which constructs a parametric response surface \rev{(with parameters modeling the randomness)} using a polynomial representation. This method exploits parametric regularity of the system to achieve fast convergence rates \cite{Db}. With gPC, stochastic solutions are represented as expansions in orthogonal polynomials of the input random parameters, and so many algorithms \rev{for parametric approximation} concentrate on computation of the expansion coefficients in a gPC representation. Collocation-based or pseudospectral-based methods are popular non-intrusive approaches to compute these coefficients, using a collection of interpolation or quadrature nodes in parameter space \cite{xiu_high-order_2005}. This requires one to query an expensive yet deterministic computational solver once for each parameter node. However, when the dimension of the random parameter is large, the size of many sampling meshes (and hence the number of computational solves) grows exponentially. This is a manifestation of the ``curse of dimensionality".

One popular strategy that combats the computational burden arising from multiple queries of an expensive model is model order reduction, which includes proper orthogonal decomposition (POD) methods, Krylov subspace methods, and reduced basis methods (RBM). The references \cite{benner_survey_2015,quarteroni_reduced_2014} detail some of these methods. Model reduction strategies allow one to replace an expensive computational model with an inexpensive yet accurate emulator for which performing a large number of queries \rev{may be} more computationally feasible.

Such an approach appeals to the same motivation as POD methods: 
\rev{although the discretized solution space is very high in dimension, the output of interest (such as the full solution field or integrated quantities of interest) frequently lie in a low-dimensional manifold. More precisely, the manifold can be approximated very well within a subspace of much lower dimension \cite{maday_priori_2002,buffa_priori_2012,PB}.} The search for, identification, and exploitation of this low-dimensional manifold are the central goals of many model order reduction strategies. Assuming such a low-dimensional manifold exists, then it may be possible to build a reduced-complexity emulator and consequently form the sought accurate gPC approximation in an efficient manner. In this paper we employ the RBM model reduction strategy for which \cite{RozzaHuynhPatera2008,PB,AT} are good references with \cite{AK,MG, Almroth_Stern_Brogan,Barrett_Reddien} the appropriate historical references. 

The Reduced Basis Method performs a projection onto \rev{a subspace spanned by} ``snapshots", i.e., a small and carefully chosen selection of the most representative high-fidelity solutions. 
\rev{The fundamental reason this is an accurate approach is that, for many PDE's of interest, the solution manifold induced by the parametric variation has small Kolmogorov width; see \cite{LorentzGolitschekMakovoz1996,pinkus_n-widths_1985} for a general discussion.}
These snapshots are selected via a \rev{weak} greedy algorithm that appeals to an \textit{a posteriori} error estimate {\cite{RozzaHuynhPatera2008, AT}}. The computational methods that one uses to compute high-fidelity snapshots include typical solvers, like spectral collocation or finite element discretizations. The ingredient in RBM that allows for computational savings is the ``offline-online" decomposition. The offline stage is the more expensive part of the algorithm where a small number \rev{(denoted $N$ throughout this paper)} of parameter values are chosen and the snapshots are generated by executing the expensive high-fidelity computational model at these parameter locations. 
The preparation completed during the offline stage allows very efficient evaluation of an emulator of the high-fidelity model at any other parameter value, i.e., ``online". During the online stage, each evaluation of the emulator can typically be computed 
\rev{orders of magnitude} 
faster than evaluation of the original expensive model. 
\rev{The actual speedup depends largely on how fast the Kolmogorov width decays and how optimally the weak greedy algorithm \cite{PB} is able to select parameter values.
Theoretical results on the quality of the $N$-dimensional surrogate space in approximating the full solution manifold are established in \cite{buffa_priori_2012} and later improved in \cite{PB}. Roughly speaking, polynomial and exponential decay rates of the theoretical Kolmogorov width of the solution manifold (with respect to $N$) can be achieved 
by the algorthmic RBM procedure. In practice, one achieves $2$ to $3$ orders of magnitude speedup \cite{RozzaHuynhPatera2008, Yc}.}
One of the major benefits of RBM that we exploit in this paper is that the RBM model reduction is \textit{rigorous}: Certifiable error bounds accompany construction of the emulator in the offline stage {\cite{AT}}.

The idea of utilizing RBM for problems in a general uncertainty quantification framework is not new \yanlai{\cite{HC,PC,PA,SB,BoyavalLeBrisEtc2009,HaasdonkUrbanWieland2013}}. In this paper our ultimate aim is to form a gPC expansion. The use of the RBM in this context, \rev{the design and analysis of an algorithm targeting statistical quantities of interest}, \yanlai{and the exploration of its effectiveness in high-dimensional random space are} underdeveloped \yanlai{to the best of our knowledge}. A \naive{} stochastic collocation or pseudospectral method is computationally challenging in high-dimensional parameter spaces, even when employing a sparse grid of economical cardinality. The hybrid gPC-RBM algorithm we propose is able to reduce the computational complexity required for construction of a gPC surrogate, and assists construction of the gPC approximation in high-dimensional parameter spaces with rigorous error bounds \rev{on the gPC coefficients, and on any Lipschitz continuous quantity of interest of the resulting expansion}.

This paper thus refines and extends the idea of combining a goal-oriented Reduced Basis Method with a generalized Polynomial Chaos expansion. Our framework is \textit{goal-oriented}: the construction of the approximation is optimized with a user-specifiable quantity of interest in mind. The algorithm is \textit{rigorous}: we can guarantee an error tolerance for general quantities of interest. Our numerical results indicate that the method improves in performance (efficiency) as the parametric dimension increases \rev{for the examples studied in this paper}. This suggests that our method is particularly useful for delaying the curse of dimensionality.

\rev{The remainder of the paper is structured as follows.} In section \ref{sec:background}, we introduce the general framework of a PDE with parametric input data. The two major ingredients in our approach, gPC and RBM, are likewise discussed. In section \ref{sec:hybrid} we introduce the novel contribution of this paper, the hybrid algorithm, \yanlai{which is analyzed in section \ref{sec:analysis}}. Our numerical results are collected in section \ref{sec:results}, and verify the efficiency and convergence of the hybrid algorithm.

\section{Background}
\label{sec:background}
In this section, we introduce the necessary background material of the hybrid algorithm, namely, generalized Polynomial Chaos and the Reduced Basis Method.  \\

\subsection{Problem setting}\label{sec:background:problem-setting}

Let $\boldsymbol{\mu} = (\mu_{1},...\mu_{K}) $ be a $K$-variate random vector with mutually independent components on a complete probability space. For $\Gamma_i = \mu_i(\Omega)$ the \rev{image} of $\mu_i$, we denote the probability density function of the random variable $\mu_{i}$ as $\rho_{i}: \Gamma_{i} \to \mathbb{R}^{+}$. Since the components of $\bmu$ are mutually independent, then $\boldsymbol{\rho}(\boldsymbol{\mu}) = \Pi_{i=1}^{K} \rho_{i}(\mu_{i})$ is the joint probability density function of random vector $\boldsymbol{\mu}.$ The image of $\boldsymbol{\mu}$ is $\boldsymbol{\Gamma} = \yc{\bigoplus}_{i=1}^{K} \Gamma_{i} \subset \mathbb{R}^{K}$.

Let $D \subset R^{d} \, (d = 1,2,3)$ be an open set in the physical domain with boundary $\partial D$, and $\boldsymbol{x} = (x_{1},...x_{d}) \in D$ \yc{a point in this set}. We consider the problem of finding the solution $u: D \times \boldsymbol{\Gamma} \to \mathbb{R}$ of the following PDE with random parameters:
\begin{align}\label{eq:pde}
  \begin{cases}\mathcal{L}(\boldsymbol{x},u,\boldsymbol{\mu}) = f(\boldsymbol{x},\boldsymbol{\mu}), \quad \forall \left(\bx, \bmu\right) \in D \times \Gamma, \\
    \mathcal{B}(\boldsymbol{x},u,\boldsymbol{\mu}) = g(\boldsymbol{x},\boldsymbol{\mu}), \quad \forall \left(\bx, \bmu\right) \in \partial D \times \Gamma.
\end{cases}
\end{align}
Here $\mathcal{L}$ is a differential operator defined on domain $D$ and $\mathcal{B}$ is a boundary operator defined on the boundary $\partial D$. The functions $f$ and $g$ represent the forcing term \yanlai{and} the boundary conditions, respectively. 

We require the problem \eqref{eq:pde} to \rev{be well-posed and have a} solution in a Hilbert space $X$. We thus assume $u(\cdot\,;\bmu) \in X$ for all $\bmu$. The Hilbert space $X$ is equipped with inner product $(\cdot,\cdot)_X$ and induced norm $\| \cdot \|_X$. A canonical example is when \eqref{eq:pde} corresponds to a linear elliptic partial differential equation \cite{Evans1998}: $X$ satisfies $H^1_0(D) \subset X \subset H^1(D)$, with $H^1$ the space of functions whose \rev{first} derivatives are square-integrable over $D$, and $H^1_0$ the space of functions in $H^1$ whose support is compact in $D$.

In most applications, one has access to a deterministic computational solver that, for each \textit{fixed} value of $\bmu$, produces an approximate, discrete solution to \eqref{eq:pde}. We assume that for this fixed $\bmu$, such a computational solver produces the discrete solution $u^{\mathcal{N}}$, which has $\mathcal{N}$ degrees of freedom. This discrete solution is obtained by solving a discretized version of \eqref{eq:pde}. For a fixed $\bmu \in \bGamma$, this is given by
\begin{align}\label{eq:truth-pde}
  \begin{cases}\mathcal{L}^{\mathcal{N}} (u^{\mathcal{N}},\boldsymbol{\mu}) = f^{\mathcal{N}}(\boldsymbol{\mu}), \\ 
  \mathcal{B}^{\mathcal{N}}(u^{\mathcal{N}},\boldsymbol{\mu}) = g^{\mathcal{N}}(\boldsymbol{\mu}).
\end{cases}
\end{align}
Standard discretizations, such as finite element or spectral collocation solvers, can be written in this way. The continuous Hilbert space $X$ is replaced with its discrete Hilbert space counterpart $X_{\mathcal{N}}$, with norm $\|\cdot\|_{X_{\mathcal{N}}}$. 

As before, we assume that $u^{\mathcal{N}}(\bmu) \in X_{\mathcal{N}}$. We will need an additional assumption that the norm of the solution is uniformly bounded as a function of the parameter. I.e., that
\begin{align}\label{eq:uniform-boundedness}
  \left\| u^{\mathcal{N}}(\cdot, \bmu) \right\|_{X_\mathcal{N}} &\leq U, & \forall \bmu &\in \bGamma
\end{align}
This assumption is satisfied for many practical problems of interest.
\yanlai{For example, for a linear elliptic operator $\mathcal{L}^{\mathcal{N}} (u^{\mathcal{N}})$, boundedness of the solution is a simple consequence of the bilinear weak form being coercive and the linear form being continuous \cite{JohnsonFEM}. In our setting, uniform coercivity of the bilinear form and uniform continuity of the linear form with respect to $\mu$ would be sufficient to guarantee the uniform boundedness \eqref{eq:uniform-boundedness}.} 

When introducing the discretized PDE \eqref{eq:truth-pde} we assume that, for any $\bmu \in \bGamma$, $u^{\mathcal{N}}(\bmu) \approx u(\bx,\bmu)$, where the approximation has an acceptable level of accuracy as determined by the modeling scenario. In practical modeling situations, one frequently requires a large number of degrees of freedom, $\mathcal{N} \gg 1$, to achieve this.

In what follows we will usually treat $\bmu$ as a parameter associated with a $\brho$-weighted norm, rather than as an explicitly random quantity. This is a standard approach, and is without loss since all of our statements can be framed in the language of probability by appropriate change of notation. \rev{For example, the space of functions of $\bmu$ with finite second moment is equivalent to the space $L^2_{\boldsymbol{\rho}}\left(\bGamma\right)$,
\begin{align*}
  L^2_{\brho}\left(\Gamma\right) &= \left\{ u: \Gamma \rightarrow \R\; | \; \int_\Gamma u^2(y) \brho(y) \dx{y} < \infty \right\} \\
                                 &= \left\{ u: \Gamma \rightarrow \R\; | \; \E u^2\left(\bmu\right) < \infty \right\}.
\end{align*}
}

\subsection{Generalized Polynomial Chaos}
The Generalized Polynomial Chaos method is a popular technique for solving stochastic PDE and representing stochastic processes \cite{Db}. The main idea of the gPC method is to seek an approximation of the exact solution of the PDE \eqref{eq:pde} by assuming that the dependence on $\bmu$ is accurately represented by a finite-degree $\bmu$-polynomial. If $u$ depends smoothly on $\bmu$, then exponential convergence with respect to the polynomial degree can be achieved. Computational implementations of gPC use an expansion in an orthogonal polynomial basis; as a consequence, quantities of interest such as expected value and variance can be efficiently evaluated directly from expansion coefficients.

\subsubsection{gPC basis}
Consider one-dimensional parameter space $\Gamma_{i}$ corresponding to the random variable $\mu_i$. If $\mu_i$ has finite moments of all orders, then there exists a collection of orthonormal polynomials $\left\{\phi^{(i)}_m(\cdot)\right\}_{m=0}^\infty$, with $\phi_m$ a polynomial of degree $m$, such that
\begin{align*}
  \E \left[ \phi^{(i)}_m(\mu_i) \phi^{(i)}_n(\mu_i) \right] = \int_{\Gamma_{i}}\rho_{i}(\mu_{i})\phi^{(i)}_{m}(\mu_{i})\phi^{(i)}_{n}\yc{(\mu_i)}d\mu_{i} = \delta_{m,n}
\end{align*}
where $\delta_{m,n}$ is the Kronecker delta function. 
The type of orthogonal polynomial basis $\{ \phi^{(i)}_{m}\}$ depends on the distribution of $\mu_i$. For instance, if $\mu_{i}$ is uniformly distributed in $[-1,1]$, its probability density function $\rho_i$ is a constant and $\{ \phi^{(i)}_{m}\}_{m=0}^\infty$ is the set of orthonormal Legendre polynomials. Several well-studied orthogonal polynomial families correspond to standard probability distributions \cite{xiu_wiener--askey_2002}. 

%

For the $K$-dimensional case ($K>1$), an orthonormal polynomial family associated to the full \rev{joint} density $\rho(\bmu)$ can be formed from products of univariate \rev{orthonormal} polynomials:
\begin{align*}
  \Phi_{\alpha}(\boldsymbol{\mu}) = \phi^{\rev{(1)}}_{\alpha_{1}}(\mu_{1})...\phi^{\rev{(K)}}_{\alpha_{K}}(\mu_{K}), 
\end{align*}
where $\alpha = (\alpha_1, \ldots, \alpha_K) \in \N_0^K$ is a multi-index. The degree of $\Phi_\alpha$ is $|\alpha| = \displaystyle {\sum_{k=1}^K \alpha_k}$. 
A standard polynomial space to consider in the multivariate setting is the \textit{total degree} space, formed from the span of all $\Phi_\alpha$ whose degree is less than a given $P \in \N$:
\begin{align}\label{eq:total-degree-definition}
  \mathcal{U}_{K}^{P} \equiv \mathrm{span} \left\{ \Phi_\alpha \;\; \big| \;\; |\alpha| \leq P \right\}
\end{align}
The dimension of \yc{$\mathcal{U}_K^P$}, \rev{denoted by $M$}, is
\begin{equation}\label{eq:total-degree-dimension}
M \equiv \mathrm{dim}(\mathcal{U}_{K}^{P}) = \left( \begin{array}{c}K+P \\ K \end{array} \right),
\end{equation}
which grows comparably to $P^K$ for large $K$. In what follows, we will index multivariate orthonormal polynomials as either $\Phi_\alpha$ with $\alpha \in \N_0^K$ satisfying $|\alpha| \leq P$, or $\Phi_m$ with $m \in \N$ satisfying $1 \leq m \leq \dim \mathcal{U}_K^P$. To achieve this, we assume any ordering of multi-indices $\alpha$ that preserves a partial ordering of the total degree (for example, graded lexicographic ordering).

\subsubsection{gPC approximation and quadrature}\label{sec:gpc-approximation}
The $L^2_\rho(\Gamma)$-optimal gPC approximation of the solution $u(\boldsymbol{x},\boldsymbol{\mu})$ to \eqref{eq:pde} in the space $\mathcal{U}^P_K$ is the $L^2_\rho(\Gamma)$-orthogonal projection onto $\mathcal{U}_{K}^{P}$, given by
\begin{equation}
\label{eq:eq1}
u_{K}^{P}(\boldsymbol{x},\boldsymbol{\mu})  =\sum_{m = 1}^{M} \widetilde{u}_{m}(\boldsymbol{x})\Phi_{m}(\boldsymbol{\mu}).
\end{equation}
The Fourier coefficient functions $\widetilde{u}_{m}$ are defined as 
\begin{align}\label{eq:uhat-definition}
\widetilde{u}_{m}(\boldsymbol{x}) = \int u(\boldsymbol{x},\boldsymbol{\mu}) \Phi_{m}(\boldsymbol{\mu})\boldsymbol{\rho}(\boldsymbol{\mu})d\boldsymbol{\mu}. 
\end{align}
For any $\bx \in D$, the mean-square error in this finite-order projection is
\begin{alignat}{1}
\label{eq:define-rho-norm}
\mathcal{E}_{\mathrm{gPC}}(\boldsymbol{x}) & = \left\|u(\boldsymbol{x},\boldsymbol{\mu}) - u^P_{K}(\boldsymbol{x},\boldsymbol{\mu}) \right\|_{L^2_{\boldsymbol{\rho}}(\boldsymbol{\Gamma})}\nonumber\\
& = \left(\int_{\boldsymbol{\Gamma}}(u(\boldsymbol{x},\boldsymbol{\mu})-u_{K}^{P}(\boldsymbol{x},\boldsymbol{\mu}))^2\boldsymbol{\rho}(\boldsymbol{\mu})d\boldsymbol{\mu}\right)^{1/2}.
\end{alignat}
Note that this error is usually not achievable in practice: The Fourier coefficients $\widetilde{u}_m$ cannot be computed without essentially full knowledge of the solution $u$. Therefore, one frequently resorts to approximating these coefficients. One popular non-intrusive method is quadrature-based pseudospectral approximation, where the integral in \eqref{eq:uhat-definition} is approximated by a quadrature rule. 

\yc{Toward that end,} let $\left\{ \bmu^q, w_q \right\}_{q=1}^Q$ denote quadrature nodes and weights, respectively, for a quadrature rule that implicitly defines a new empirical \yc{probability} measure:
\begin{align}\label{eq:dense-quadrature-rule}
  \int_{\bGamma} f(\bmu) \rho(\bmu) \dx{\bmu} \approx \sum_{q=1}^Q w_q f\left(\bmu^q\right).
\end{align}
For example, two common choices for quadrature rules are tensor-product Gauss quadrature rules, and Gauss-Patterson-based sparse grid quadrature rules (e.g., \cite{conrad_adaptive_2013}). Each of these rules can integrate polynomials of high degree, but the requisite size of the quadrature rule $Q$ is large in high dimensions. \rev{Both quadrature grids have cardinality that grows exponentially with dimension (for a fixed degree of integration accuracy), although sparse grids have a smaller size among the two. 
In this paper, we use tensor-product Gauss quadrature rules for low dimensions and sparse grid constructions for high dimensions; with these choices, we hereafter assume that \eqref{eq:dense-quadrature-rule} holds for functions $f$ of the form of the integrand in \eqref{eq:uhat-definition}.}


With this quadrature rule, the Fourier coefficients can be approximated by 
\begin{align}\label{eq:uhat-approximation}
  \widetilde{u}_m \approx \widehat{u}_m = \sum_{q=1}^Q u\left(\bx, \bmu^q\right) \Phi_m(\bmu^q) w_q.
\end{align}
The advantage of this formulation is that we need only compute the quantities $u\left(\cdot, \bmu^{\yc{q}}\right)$, which are a collection of solutions to a deterministic PDE. Since this is all done in the context of a computational solver given by \eqref{eq:truth-pde}, one will replace the continuous solution $u(\cdot, \bmu)$ with the discrete solution $u^{\mathcal{N}}(\bmu)$.

Then a straightforward stochastic quadrature approach \yc{first} collects the solution ensemble from the computational solver, and subsequently uses it to compute approximate Fourier coefficients:
\begin{align}\label{eq:uhat-truth-approximation}
  \left\{ u_q^{\mathcal{N}}\left( \bx \right) \right\}_{q=1}^Q \coloneqq \left\{ u^{\mathcal{N}}\left( \bx, \bmu^q\right)\right\}_{q=1}^Q \hskip 10pt \longrightarrow \hskip 10pt
  \widehat{u}^{\mathcal{N}}_m = \yc{\sum_{q=1}^Q u_q^{\mathcal{N}} \left(\bx \right) \Phi_m(\bmu^q) w_q,} 
\end{align}
The full approximation is formed by replacing the exact Fourier coefficents $\widetilde{u}_m$ in \eqref{eq:eq1} with the $\widehat{u}^{\cN}_m$ coefficients computed above.

Note that, in order for the quadrature approximation \eqref{eq:uhat-truth-approximation} to be reasonably accurate, the number of quadrature points $Q$ should be comparable with $M$. We already know from \eqref{eq:total-degree-dimension} that $M$ scales like $P^K$ for total-degree spaces. 
Typically, the cost of obtaining each \yanlai{$u^{\mathcal{N}}_q$} requires at least $\mathcal{O}(\mathcal{N})$ computational effort. (In some cases $\mathcal{O}(\mathcal{N}^3)$ effort is required.) \yc{$Q$} solves of the PDE are required, with each solve costing at least ${\mathcal{O}}(\mathcal{N})$ work. Since $Q \sim M \sim P^K$, then in the best-case scenario the total work scales like $\mathcal{O}\left(\mathcal{N} P^K \right)$. Thus, the requisite computational effort for a straightforward stochastic quadrature method is onerous when \rev{$K$, the dimension of the random parameter $\bmu$, is large}.

However, if these coefficients could be computed, then the resulting expansion \eqref{eq:eq1} can be very accurate. The focus of this paper is to \textit{inexpensively} compute an approximation to $\widehat{u}^{\cN}_m$. The essential ingredient is replacement of $u^{\mathcal{N}}_{\yc{q}}$ by a surrogate that is much cheaper to compute, and whose approximation fidelity can be rigorously quantified.


\subsubsection{Quantities of Interest}

In many UQ scenarios, one is not necessarily interested in the entire solution field $u(\bx,\bmu)$, but rather some other quantity of interest derived from it. We introduce a functional $\mathcal{F}$ that serves to map the solution $u$ to the quantity of interest (the ``goal"). 
Our construction exploits the well-known property of gPC that common quantities of interest such \yanlai{as} the mean field and variance field can be recovered by simple manipulation of the gPC coefficients \cite{Db}, up to the accuracy of the gPC expansion.
Our theoretical results require two assumptions on the quantity of interest map $\mathcal{F}$.

The first assumption we make is that $\mathcal{F}$ has affine dependence on an $M$-term gPC expansion, $u_M$, specifically
\begin{align}
\label{eq:F-affine}
  \mathcal{F} \left[ u_M \right] = \mathcal{F} \left[ \sum_{m=1}^M \widehat{u}_m(\bx) \phi_m(\bmu) \right] = \sum_{m=1}^M \theta_{\mathcal{F}}(\widehat{u}_m\left(\bx\right)) \mathcal{F} \left[ \phi_m(\bmu)\right],
\end{align}
for some function $\theta_{\mathcal{F}}$. It is not hard to show that typical quantities of interest satisfy this condition on $\mathcal{F}$ with simple coefficient functions $\theta_{\mathcal{F}}$:
\begin{subequations}\label{eq:F-examples}
\begin{itemize}
\item $\mathcal{F}$ is the expected value operator $\E$ and $\theta_{\mathcal{F}}$ is the identity function, 
\begin{align}
\label{eq:mean-value}
\mathcal{F}\left[ u_M \right] &= \E \left[ u_M(\bx, \bmu) \right] = \sum_{m=1}^M \widehat{u}_m(\bx) \E \left[ \phi_m(\bmu) \right] = \widehat{u}_{1}\rev{(x)}
\end{align}
\item $\mathcal{F}$ is the variance or norm-squared operator and $\theta_{\mathcal{F}}$ is the quadratic function $\theta_{\mathcal{F}}(v) = v^2$,
\begin{align}
\label{eq:variance}
 \mathcal{F}\left[ u_M \right] &= \mathrm{var} (u_M(\bx, \bmu) )= \sum_{m=1}^M (\widehat{u}_m(\bx))^2 \mathrm{var} \left[ \phi_m(\bmu) \right] = \sum_{m=2}^M (\widehat{u}_m(\bx))^2\\
 \label{eq:rho-norm}
 \mathcal{F}\left[ u_M \right] &= ||u_M(\bx, \bmu)||^2_{L_{\rho}^2} = \sum_{m=1}^M (\widehat{u}_m(\bx))^2  || \phi_m(\bmu) ||_{L_{\rho}^2}= \sum_{m=1}^M (\widehat{u}_m(\bx))^2.
\end{align}
\end{itemize}
\end{subequations}

Our theoretical results also require a second assumption: that the functional $\theta_{\mathcal{F}}$ is Lipschitz continuous with Lipschitz constant $C_{\mathrm{Lip}}$, i.e., that
\begin{align}\label{eq:F-lipschitz}
  \left\| \theta_{\cF}(v) - \theta_{\cF}(w) \right\| \leq C_{\mathrm{Lip}} \left\| v - w \right\|,
\end{align}
for all appropriate inputs $v$ and $w$. For $\mathcal{F} = \E$, this constant is $C_{\mathrm{Lip}} = 1$. For the latter cases of $\mathcal{F} = \mathrm{var}$ and $\mathcal{F}\left[ \cdot\right] = \left\| \cdot \right\|_{L^2_\rho}$ where $\theta_{\mathcal{F}}(v) = v^2$, then $C_{\mathrm{Lip}} = 2 U$ with $U$ is the uniform bound in \eqref{eq:uniform-boundedness}.  \rev{This is a consequence of $\left\| \theta_{\cF}(v) - \theta_{\cF}(w) \right\| = \left\| v^2 - w^2 \right \| \le \left\| v + w \right\| \cdot \left\| v - w \right\|$ and the uniform-boundedness of the solution \eqref{eq:uniform-boundedness}}.

\subsection{Reduced Basis Method (RBM)}\label{sec:background-rbm}
The reduced basis method \cite{RozzaHuynhPatera2008,PB,AT} is a popular model order reduction \yc{strategy} to solve a parameterized PDE for a large number of different parameter configurations. RBM seeks to form an approximation $u^N$ satisfying
\begin{align*}
  u^N(\bmu) &\approx u^{\mathcal{N}}(\bmu), & \bmu &\in \bGamma,
\end{align*}
such that the approximation $u^N$ can be computed with an algorithm whose complexity depends only on $N$, in contrast to the full solution $u^{\mathcal{N}}$ whose complexity depends on $\mathcal{N} \gg N$. 
\rev{To achieve this speedup, RBM algorithms traditionally make 3 assumptions:
  \begin{itemize}
    \item The Kolmogorov $N$-width of the (discretized) solution manifold is small, i.e.,
      \begin{align}\label{eq:n-width}
      d_N \coloneqq \inf_{\dim Y = N} \sup_{\bs{\mu} \in \bs{\Gamma}} \inf_{v \in Y} \left\| u^{\cN}\left(\cdot, \bs{\mu}\right) - v \right\|_{X^{\cN}}
      \end{align}
      is small when $N \ll \cN$, where the outer infinimum is taken over all $N$-dimensional subspaces of $X^{\cN}$. In practice one hopes that $d_N$ decays algebraically or even exponentially with increasing $N$. The small $N$-width requirement is a fundamental mathematical assumption without which RBM cannot achieve acceptable error with small $N$. RBM forms a dimension-$N$ space $X^N$ that seeks to achieve a distance to the solution manifold that is close to $d_N$.
    \item There is an efficiently-computable rigorous \textit{a posteriori} error estimate. Given the $N$-dimensional RBM subspace $X^N$ of $X^{\cN}$, an estimate $\Delta_N(\bs{\mu})$ can be computed with $\mathcal{O}(N)$ complexity, satisfying
      \begin{align}\label{eq:error-estimate}
         \Delta_N \left(\bs{\mu}\right) &\geq  \left\| u^{\cN}\left(\cdot, \bs{\mu}\right) - u^N(\mu)\right\|_{X^{\cN}}, & \bs{\mu} &\in \bs{\Gamma}
      \end{align}
      The computable error estimate is required for an iterated, greedy construction of RBM approximation spaces $X^N$.
    \item The operators $\mathcal{L}$ and $f$ in \eqref{eq:pde} have affine dependence on $\bs{\mu}$. I.e., there exist $\bs{\mu}$-independent operators $\bbm{L}_q$ and $f_q$, and $\bs{x}$-independent functions $\theta^L_q(\bs{\mu})$ and $\theta^f_q(\bs{\mu})$ such that 
      \begin{align}\label{eq:affine-assumptions}
          \mathcal{L}(\boldsymbol{\mu}) &= \sum_{q = 1}^{Q_{L}}\boldsymbol{\theta}_{q}^{L}(\boldsymbol{\mu})\mathbb{L}_{q}, &
          f(\boldsymbol{x},\boldsymbol{\mu}) &= \sum_{q = 1}^{Q_{f}}\boldsymbol{\theta}_{q}^{f}(\boldsymbol{\mu})f_{q}(\boldsymbol{x}) 
      \end{align}
      The affine assumptions are required so that RBM can compute the reduced-order solution $u^N$ with $N \ll \mathcal{N}$ complexity.
  \end{itemize}
  In this paper, we will also assume that $\mathcal{L}$ is a linear operator: this simplifies presentation of some RBM mechanics and allows easy motivation of formulas connecting PDE residuals with solution errors. While these assumptions are traditional, there are constructive remedies for non-affine and \rev{certain} non-linear operators {\cite{MB, MA, NC}}. 
}
The RBM algorithm is a central part of the novel hybrid approach that we present in Section \ref{sec:hybrid}.

\subsubsection{Reduced basis approximation}

We recall from the discussion in Section \ref{sec:background:problem-setting} that a computational solver in \eqref{eq:truth-pde} uses $\mathcal{N} \gg 1$ degrees of freedom to produce $u^{\mathcal{N}}_{\yc{q}}(\bx)$, which is deemed an acceptably accurate approximation to $u\left(\bx, \bmu^{\yc{q}}\right)$. In the RBM context, this approximation is called the \textit{truth solution} or \textit{\yc{truth approximation}} and we will use this terminology when appropriate. 
The starting point for developing computational reduced basis methods is to replace the expensive truth solution with an inexpensive reduced-order solution. We briefly describe the standard method for accomplishing this below. 

Assume that a training set of parameter samples $\Xi \in \bGamma$ is given such that the $\bmu$-variation of the solution $u$ is accurately captured by the resulting truth solution ensemble \yc{\{$u^{\mathcal{N}}(\bx, \bmu): \bmu \in \Xi$\}}. \footnote{More precisely, we require that manifold of the finite ensemble of solutions $u^{\mathcal{N}}\left(\cdot, \Xi\right)$ is an accurate surrogate for the full manifold $u^{\mathcal{N}}\left(\cdot, \bGamma\right)$.}
\rev{This is a rather stringent requirement on $\Xi$, and does not furnish a constructive definition in general. }
In this paper we take $\Xi$ to be the quadrature rule nodal set introduced in \eqref{eq:dense-quadrature-rule}, \yc{that is, $\Xi = \left\{ \bmu^q \right\}_{q=1}^Q$}. \rev{Since our ultimate goal is formation of a gPC surrogate, we argue that such a choice of $\Xi$ is a reasonable choice for the RBM training set.}

For any given reduced-order dimension $N \ll \mathcal{N}$, we build the $N$-dimensional reduced basis space $X^{N}$ by a greedy algorithm. This space is constructed as a span of ``snapshots" (i.e., truth solutions) \cite{Maday}
\begin{align}\label{eq:XN-space}
  X^{N} = \mathrm{span} \{ u^{\mathcal{N}}(\boldsymbol{x}, \bnu^1), ..., u^{\mathcal{N}}(\boldsymbol{x}, \bnu^{N})\},
\end{align}
where $\left\{\bnu^1, \ldots, \bnu^N\right\} \subset \Xi$ are chosen in a greedy fashion. Assuming the greedy algorithm is performed in a sufficiently accurate fashion, then the dimension-$N$ space approximates the manifold $u\left(\cdot, \bs{\Gamma}\right)$ with an error comparable to the $N$-width \cite{PB}. The small $N$-width assumption \eqref{eq:n-width} thus guarantees that a greedily-constructed $X^N$ achieves a small approximation error. Existence of an efficiently-computable error estimate satisfying \eqref{eq:error-estimate} allows a feasible greedy search for the $\bs{\nu}^k$ to be performed. (See sections \ref{sec:rbm-error-estimate} and \ref{sec:rbm-greedy} below.)

\yc{For a fixed $\bmu^\ast \in \bGamma$, RBM approximates the truth solution $u^{\mathcal N}(\bx, \bmu^\ast)$ by an element from $X^N$.} Thus, the RBM surrogate $u^{N}(\boldsymbol{x},\boldsymbol{\mu}^{*})$ can be represented as
\begin{align}\label{eq:uN-ansatz}
u^{N}(\boldsymbol{x},\boldsymbol{\mu}^{*}) = \sum_{k = 1}^{N} c_{k}(\boldsymbol{\mu}^{*})u^{\mathcal{N}}(\boldsymbol{x},\boldsymbol{\mu}^{k})
\end{align}
RBM algorithms proceed by computing the $c_k\left(\bmu^{*}\right)$ so that the PDE residual with the ansatz \eqref{eq:uN-ansatz} is as small as possible. The meaning of ``small" is made precise by the prescription of an appropriate projection operator $\mathcal{P}$ such that, using linearity of operator, the following holds:
\begin{equation}\label{eq:rbm-projection}
  \mathcal{P} \left[\sum_{k = 1}^{N}c_{k}(\boldsymbol{\mu}^{*})\mathcal{L}^{\mathcal{N}}(\boldsymbol{x}, u^{\mathcal{N}}(\boldsymbol{x},\boldsymbol{\mu}^{k}), \boldsymbol{\mu}^{*}) \right] = \mathcal{P} \left[ f^{\mathcal{N}}(\boldsymbol{x},\boldsymbol{\mu}^{*}) \right].
\end{equation}
Concrete examples of this abstract projection operator are the continuous $L^2$ projection onto $X^N$, a discrete $\ell^2$ projection (least-squares) on the spatial mesh, or an empirical interpolation procedure \cite{Yc,ChenGottliebMaday}. The affine assumption \eqref{eq:affine-assumptions} allows computation of the $c_k\left(\bmu^*\right)$ with an $N$-dependent complexity (as opposed to $\cN \gg N$ complexity): Using \eqref{eq:affine-assumptions} in \eqref{eq:rbm-projection}, we have
\begin{align}\label{eq:rbm-residual-projection}
  \mathcal{P} \left[ \sum_{k = 1}^{N} c_{k}(\boldsymbol{\mu}^{*}) \sum_{q = 1}^{Q_{L}} \boldsymbol{\theta}_{q}^{L}(\boldsymbol{\mu}^{*}) \dul{\mathbb{L}_{q}(u^{\mathcal{N}}(\boldsymbol{x},\boldsymbol{\mu}^{k}))} \right]
  = \mathcal{P} \left[ \sum_{q = 1}^{Q_{f}} \boldsymbol{\theta}_{q}^{f}(\boldsymbol{\mu}^{*}) \dul{f_{q}(\boldsymbol{x})}\right]
\end{align}
We note that only the terms that are double-underlined require $\mathcal{N}$-dependent complexity to evaluate. However, these terms do \textit{not} depend on $\bmu^\ast$, and so they may be computed and stored during the offline stage. We refer to \cite{Yc, GR} for more details.

In the following sections we detail portions of the RBM algorithm that our novel algorithm amends: the error estimate and the greedy algorithm (sections \ref{sec:rbm-error-estimate} and \ref{sec:rbm-greedy}, respectively).

\subsubsection{\textit{A posteriori} error estimate}\label{sec:rbm-error-estimate}
\rev{The goal of this section is to compute a bound on the norm of $e(\boldsymbol{x},\boldsymbol{\mu}) \coloneqq u^{\mathcal{N}}(\boldsymbol{x},\boldsymbol{\mu}) -u^{N}(\boldsymbol{x},\boldsymbol{\mu})$ without computing the truth solution $u^{\mathcal{N}}$. Let $R_N : D \times \bs{\Gamma}$ denote the (Riesz representation of the) truth discretization residual from \eqref{eq:truth-pde} using the reduced-order solution from $X_N$:
  \begin{align}\label{eq:error-equation}
    R_N \left(\bx, \bmu\right) \coloneqq f^{\mathcal{N}}(\cdot\,;\boldsymbol{\mu}) - \mathcal{L}^{\mathcal{N}}(u^{N}(\boldsymbol{x}, \boldsymbol{\mu});\boldsymbol{\mu}) = \mathcal{L}^{\mathcal{N}}(e(\boldsymbol{x},\boldsymbol{\mu}),\boldsymbol{\mu})
  \end{align}
}
\rev{Let $\beta_{LB}(\boldsymbol{\mu})$ be a lower bound for the smallest eigenvalue of $\mathcal{L}^{\mathcal{N}}(\boldsymbol{\mu})^{T}\mathcal{L}^{\mathcal{N}}(\boldsymbol{\mu})$:
\begin{equation}\label{eq:beta-lb}
  0 < \beta_{LB}(\boldsymbol{\mu}) \leq \min_{v} \yc{\frac{v^{T}\left(\mathcal{L}^{\mathcal{N}}(\boldsymbol{\mu})\right)^{T} \mathcal{L}^{\mathcal{N}}(\boldsymbol{\mu})v}{\lVert v \rVert_{X_{\mathcal N}}}} 
\end{equation}
}
\yc{Above, $\mathcal{L}^{\mathcal{N}}(\boldsymbol{\mu})$ should be understood as the matrix} representation of the operator ${\mathcal L}^{\mathcal N}(\cdot\,; \bmu)$. The relations \eqref{eq:error-equation} and \eqref{eq:beta-lb} can be used to conclude \cite{Yc} 
\begin{equation} 
\yc{\Delta_{N}(\boldsymbol{\mu}) \coloneqq \frac{{\lVert R_N(\boldsymbol{x}, \boldsymbol{\mu})\rVert_{X_{\mathcal N}}}}{\sqrt{\beta_{LB}(\boldsymbol{\mu})}} \geq \lVert e(\boldsymbol{x},\boldsymbol{\mu}) \rVert_{X_{\mathcal N}}  \qquad \forall \boldsymbol{\mu} \in \Xi}
\end{equation}
\rev{Thus the \textit{a posteriori} error estimate $\Delta_N$ is rigorous, satisfying \eqref{eq:error-estimate}. For the computation of this bound to be efficient, we must compute the residual $R_N$ and $\beta_{LB}$ in an $\mathcal{N}$-independent fashion. 
  The residual can be computed with $\mathcal{O}(N)$ complexity by exploiting the same manipulations used in \eqref{eq:rbm-residual-projection}. The efficient evaluation of $\beta_{LB}(\boldsymbol{\mu})$ can be accomplished via the successive constraint linear optimization method (SCM) {\cite{RozzaHuynhPatera2008,HuynhSCM, HKCHP, Chen2015_NNSCM}} with \yc{the marginal} computational cost \yc{for each $\bmu$} independent of the truth solution complexity $\mathcal{N}$. (There is a one-time $\cN$-dependent overhead computation.) 
  
  With the ability to efficiently compute $\Delta_N$, we can describe the greedy algorithm for choosing the RBM parameter snapshot locations \yc{$\{\bnu^k\}$}. 
}

\subsubsection{Greedy Algorithm}\label{sec:rbm-greedy}
\rev{Given a current set of parametric samples $\left\{ \bs{\nu}^1, \ldots, \bs{\nu}^k\right\}$ and the training parameter set $\Xi$, a new parameter value $\boldsymbol{\nu}^{\yc{k+1}} \in \Xi$ is ideally selected as the $\bmu \in \Xi$ that maximizes the error between $u^{\cal{N}}\left(\cdot,\bs{\mu}\right)$ and its projection onto $X^k$. Computationally, this error may be estimated by $\Delta_k$:}
\begin{equation}\label{eq:greedy-selection}
\boldsymbol{\nu}^{k+1} = \argmax_{\boldsymbol{\mu} \in \Xi} \Delta_{k}(\boldsymbol{\mu})
\end{equation}
This process is repeated either until the maximum value of $\Delta_N$ is \rev{smaller than a user-prescribed tolerance $\epsilon_\mathrm{tol}$, or until $N$ reaches a user-defined maximum value.} 

\rev{In total, only $N$ queries of the truth solution are required. (One query for each parameter selected via \eqref{eq:greedy-selection}.) However, each optimization \eqref{eq:greedy-selection} requires sweeping over $\Xi$, evaluating the error estimate $\Delta_k$ everywhere in the training set. Although we have described above that evaluation of $\Delta_k$ has a complexity that is only $N$-dependent, the cardinality of $\Xi$ can be very large, and so this greedy optimization can still be expensive for large parametric dimensions.}

\rev{We close this section by remarking that there are several algorithmic optimizations to speed up the computation in \eqref{eq:greedy-selection}. For example, the evaluation of $\Delta_k(\bmu)$ for each $\bmu \in \Xi$ is embarrassingly parallel. In addition, $\Delta_k$ is monotonically decreasing with $k$ if the projection operator $\mathcal{P}$ in \eqref{eq:rbm-residual-projection} is an orthogonal projection. In this case, any $\bmu \in \Xi$ satisfying $\Delta_k(\bmu) < \epsilon_{\mathrm{tol}}$ may be permanently removed from future greedy sweeps.}

\section{Hybrid Algorithm}\label{sec:hybrid}
Recall that straightforward stochastic quadrature requiring $Q$ solves of \eqref{eq:truth-pde}, each having $\mathcal{N}$-dependent complexity, can be computationally burdensome when the random parameter dimension $K$ is large.  
Our approach simply replaces $u^{\cN}$ in \eqref{eq:uhat-truth-approximation} with an RBM surrogate $u^N$ that is tailored toward gPC approximations. It thus ameliorates the cost-per-solve (reducing it to an $N$-dependent operation count) by using a gPC-goal-oriented variant of the RBM algorithm.

Our algorithm uses a modified {\em a posteriori} error estimate in the traditional RB greedy algorithm. This results in a RBM-gPC hybrid algorithm that can accurately construct a gPC surrogate more efficiently than standard pseudospectral methods. 

\rev{Our convergence metric is the $L^2_\rho$ norm, and so we modify the standard RBM algorithm described in section \ref{sec:background-rbm} so that rigorous $L^2_\rho$ error estimates may be derived.}
We exploit the observation that each $u(\bx, \bmu^q)$ associated with parameter value $\bmu^q$ should have some quantitative measure of importance as indicated by the probability density $\rho\left(\bmu^q\right)$. This idea was explored in \cite{PC}, but our version differs notably from earlier methods since we do not explicitly use $\rho$ as a weight for the \textit{a posteriori} error estimate. 


\subsection{Weighted {\em a posteriori} error estimate}
Design of the RBM error estimate can significantly affect the performance of \yc{the resulting reduced basis method, particularly so for our} goal-oriented approach. \yc{The approximate gPC coefficient formula \eqref{eq:uhat-approximation} is the $w_q$-weighted inner product between the polynomial $\Phi_m$ and $u^{\mathcal{N}}$. Thus, the error estimate should likewise be weighted using the quadrature weight $w_q$.} We emphasize again that our strategy is different from using the probability density function $\rho$ as done in \cite{PC}; even in simple one-dimensional cases, it is easy to see that $w_q \not\propto \rho\left(\bmu^q\right)$ (e.g., Gaussian quadrature or Clenshaw-Curtis quadrature).

We introduce the following 
weighted \textit{a posteriori} error estimate $\Delta^w_{N}(\boldsymbol{\mu})$:
\begin{align}\label{eq:weighted-aposteriori-estimate}
  \yc{\Delta^w_N}\left(\bmu^q\right) &= \yc{\frac{\left\|R_N(\cdot, \bmu^q) \right\|_{X_{\mathcal N}}}{\sqrt{\beta_{\mathrm{LB}}(\bmu^q)}}} \sqrt{Q |w_q|}, & q &= 1, \ldots, Q,
\end{align}
where $\beta_{LB}(\boldsymbol{\mu}^q)$ is a lower bound for the smallest eigenvalue of $\mathcal{N}^\cN$ as given in \eqref{eq:beta-lb}, and $R_N$ is the PDE residual of the order-$N$ surrogate as defined in \eqref{eq:error-equation}. Note that the novel quantity is the factor $\sqrt{Q |w_q|}$; since $w_q$ corresponds to a $Q$-point normalized quadrature rule, the quantity $Q |w_q|$ frequently has $\mathcal{O}(1)$ magnitude. \rev{For example, a univariate Gauss quadrature rules on compact domains $Q w_q \sim \sqrt{1-\left(\mu^q\right)^2} \rho(\mu^q) \sim \mathcal{O}(1)$ for fairly general $\rho$ \cite{Paul}.}  The absolute value bars in \eqref{eq:weighted-aposteriori-estimate} are necessary in general because sparse grid quadrature rules can have negative weights. For a tensor-product Gaussian quadrature rule, the weights are all positive.

The $m$-th Fourier coefficient produced by the hybrid algorithm is 
$\widehat{u}^{\mathcal{N}}_m$ in \eqref{eq:uhat-truth-approximation} and its surrogate  $\widehat{u}^{N}_m$
\begin{align}\label{eq:uhat-truth-surrogate}
  \widehat{u}^{{N}}_m = {\sum_{q=1}^Q u^{{N}}(\bmu^q) \Phi_m(\bmu^q) w_q.}.
\end{align}
The form of the weighted {\em a posteriori} error estimate ensures that we can bound the error between this coefficient and the corresponding truth Fourier coefficient. The precise estimate appears \yanlai{in section \ref{sec:analysis}}.

\subsection{Goal-oriented greedy algorithm}
\rev{The greedy algorithm strategy here is essentially the same as in Section \ref{sec:rbm-greedy}. One major difference is that we replace the original error estimate with our weighted one:}
\begin{equation}
\boldsymbol{\nu}^{k+1} = \argmax_{\boldsymbol{\mu} \in \Xi} \Delta^w_{k}(\boldsymbol{\mu})
\end{equation}
We show later in Theorem \ref{thm:qoi-convergence} and Corollary \ref{corollary:qoi-convergence-bound} that the weighting provided by \rev{$\Delta^w_k$} allows one to guarantee that the gPC approximation that is formed from the RBM \yc{surrogates} is within a user-defined tolerance of the gPC-truth approximation.

Another major difference in the goal-oriented algorithm is that the tolerance criterion is tuned to the quantity of interest of the gPC surrogate. At each stage with $k$ RB snapshots, we compute the error estimate
\begin{align}\label{eq:epsilon-def}
  \varepsilon &= C_{Q,M} \sqrt{\frac{1}{Q} \sum_{q=1}^Q \yc{(\Delta^w_k)^2}(\bmu^q) }, \mbox{ with }  C_{Q,M} = \sum_{m=1}^M \sqrt{\sum_{q=1}^Q |w_q| \Phi_m^2(\bmu^q)} \left|\mathcal{F}\left[ \Phi_m(\bmu)\right]\right|,
\end{align}
where $\mathcal{F}$ denotes the quantity of interest as introduced in \eqref{eq:F-affine}. Note that the constant $C_{Q,M}$ is computable \textit{independent} of the solution $u$, and depends only on the choice of quadrature rule and quantity of interest. (See Lemma \ref{lemma:BQm-gauss-quadrature}, and the discussion following Corollary \ref{corollary:qoi-convergence-bound}.) As we show in Corollary \ref{corollary:qoi-convergence-bound}, $\varepsilon$ is an upper bound on the error in the quantity of interest defined by $\mathcal{F}$ between the inexpensive RBM surrogate and the expensive gPC truth approximation.

A pseudocode implementation of the weighted greedy approach is shown in Algorithm \ref{alg:goal-oriented-greedy}.

\begin{algorithm}[h!]
\caption{Goal-oriented greedy algorithm}\label{alg:goal-oriented-greedy}
\begin{algorithmic}[1]
\State Input: training set $\Xi$ with associated quadrature weights $w_q$;
\State Input: stopping criterion tolerance $\varepsilon_{\mathrm{tol}}$;
\State Input: goal-oriented constant $C_{Q,M} = \sum_{m=1}^M B_{Q,m} \mathcal{F}\left[ \Phi_m\left(\bmu\right)\right]$.
\State Randomly select the first sample $\boldsymbol{\mu^{1}} \in \Xi$;
\State Obtain truth solution $u^\mathcal{N}\left(\bx, \bmu^1\right)$, set $X_1 = \mathrm{span}\left\{ u^{\mathcal{N}}(\boldsymbol{x},\boldsymbol{\mu}^{1})\right\}$;
\State Set $k = 1$ and $\varepsilon = \infty$;
\vspace{0.1in}
\While {$\varepsilon > \varepsilon_{\mathrm{tol}}$}
\vspace{0.05in}
\For{each $\mu \in \Xi$}
\State Obtain RBM solution $u^k(\bx, \bmu)$ by computing $c_j(\bx)$ that satisfy~\eqref{eq:rbm-projection}
\State Compute weighted \textit{a posteriori} error estimate $\yc{\Delta^w_{k}}(\boldsymbol{\mu})$ from \eqref{eq:weighted-aposteriori-estimate}
\EndFor
\vspace{0.05in}
\State Choose $\boldsymbol{\mu}^{k+1} = \argmax_{(\boldsymbol{\mu} \in \Xi)} \Delta^w_{k}(\boldsymbol{\mu})$;
\vspace{0.05in}
\State augment the reduced basis space $X_{k+1} = \mathrm{span} \left\{X_{k} \cup \left\{ u^{\mathcal{N}}(\boldsymbol{x},\boldsymbol{\mu}^{k+1})\right\}\right\}$;
\vspace{0.05in}
\State Calculate $\Delta^{\mathrm{sum}} = \sum_{\boldsymbol{\mu} \in \Xi} (\Delta^w_{N})^2(\boldsymbol{\mu})$; 
\vspace{0.05in}
\State Set $\varepsilon = C_{Q,M} \sqrt{\frac{1}{|\Xi|} \Delta^{\mathrm{sum}}}$.;
\State Set $N \gets N+1$;
\vspace{0.05in}
\EndWhile
\vspace{0.1in}
\end{algorithmic}
\end{algorithm}

\subsection{Goal-oriented hybridized RBM-gPC algorithm}

Summary pseudocode of the goal-oriented gPC-RBM procedure is shown in Algorithm \ref{euclid}. 

The offline stage is Algorithm \ref{alg:goal-oriented-greedy}: the RBM approximation space is built by scanning the quadrature node set $\Xi$ and evaluating the weighted \textit{a posteriori} error estimate $\Delta^w_k$. The expensive PDE solver \eqref{eq:truth-pde} is queried a total of $N$ times over the greedily constructed parameter set $\bnu^1, \ldots, \bnu^N$. The offline phase stops when the computed error indicator $\varepsilon$ defined in \eqref{eq:epsilon-def} falls below the user-defined tolerance $\varepsilon_{\mathrm{tol}}$.

The online stage proceeds as indicated on line \ref{alg:euclid:online} of Algorithm \ref{euclid}. In this phase, the RBM surrogate that was constructed in the offline phase is evaluated several times to compute the approximate gPC coefficients $\left\{ \widehat{u}^N_{m} \right\}_{m=1}^M$. This portion of the algorithm is \textit{much} faster than \naive{} evaluation of \eqref{eq:uhat-truth-approximation}; it is in this section of the procedure where the hybrid algorithm reaps computational savings compared to a traditional stochastic pseudospectral approach.

Once the approximate coefficients $\widehat{u}^N_m$ are collected, the quantity of interest $\mathcal{F}\left[u^N_M\right]$ may be evaluated. A significant benefit of using the hybrid approach is that the error in this computed quantity of interest can be rigorously controlled by the user-defined input tolerance $\varepsilon_{\mathrm{tol}}$; we show this in the next section.

Thus, the hybrid algorithm \textit{both} achieves significant computational savings in construction of a gPC approximation \textit{and} provides strict error bounds on quantities of interest.

\begin{algorithm}
\caption{Hybrid Algorithm}\label{euclid}
\begin{algorithmic}[1]
  \State Input: the general stochastic PDE \eqref{eq:pde} and a tolerance $\varepsilon_{\mathrm{tol}}$.
\vspace{0.1in}
\State{\bf{Offline procedure:}}
\vspace{0.05in}
\State Set $\Xi$ as the training set and use the goal-oriented greedy procedure with tolerance $\varepsilon_{\mathrm{tol}}$ (Algorithm \ref{alg:goal-oriented-greedy}) to compute $N$ reduced basis elements $\{u^{\mathcal{N}}(\boldsymbol{x},\boldsymbol{\mu}^{j})\}_{j = 1}^{N}$.
 \vspace{0.05in}
\State{\bf{end Offline procedure}}
\State Using the online RBM surrogate $u^N(\bx,\bmu)$, compute $M$-term gPC coefficients and approximation of $u^N$:\label{alg:euclid:online}
\begin{align*}
  \widehat{u}^N_m &= \sum_{q=1}^Q w_q u^N\left(\bx, \bmu^q\right) \Phi_m\left(\bmu^q\right)\\
  u_M^N(\boldsymbol{x},\boldsymbol{\mu}) &= \sum_{m = 1}^{M} \hat{u}^N_{m}(\boldsymbol{x})\Phi_{m}(\boldsymbol{\mu}).
\end{align*}
\State Output: Solution $u^N_M$, or quantity of interest $\mathcal{F} [u^N_M] = \sum_{m=1}^M \theta_{\mathcal{F}}[\widehat{u}_m^N(\bx)] \mathcal{F}\left[ \Phi_m(\bmu)\right]$.
\vspace{0.1in}
 \vspace{0.1in}
\end{algorithmic}
\end{algorithm}

\rev{\subsubsection{Computational cost}
The computational complexities of each step of the hybrid algorithm are as follows. Let $\cS(\cN)$ denote the 
cost of one truth solve; depending on the problem and the solver employed, this typically varies between $\mathcal{O}(\cN)$ and $\mathcal{O}\left(\cN\,^3\right)$. 

The offline portion of the algorithm requires repeated scanning of parameter space for maximization of $\Delta^w_k$. A \naive{} scanning requires $Q$ evaluations of $\Delta^w_k$. 
However, we can be more efficient. For example, we may trim entries in the set $\{\bmu \in \Xi: \Delta^w_{N}(\bmu) < \frac{\epsilon_{\rm tol}}{2 C_{Q,M}}\}$ from $\Xi$ after each loop. We let $n(\Xi, k)$ denote the size of the set 
\begin{align*}
  n(\Xi, k) &= \left| \left\{ \bmu \in \Xi \; | \; \Delta^w_{k-1} \leq \frac{\epsilon_{\rm tol}}{2 C_{Q,M}} \right\} \right|, & \Delta^w_0(\bmu) &\equiv \infty
\end{align*}
Obviously $n(\Xi, k) \le |\Xi| \equiv Q$, and $n(\Xi,k+1) \leq n(\Xi, k)$. Finally, we define $Q^{\rm trim}_N \coloneqq \sum_{k=1}^Nn(\Xi,k)$. Using this notation we can give a rough operation count of the entire hybrid gPC-RBM algorithm.}

\rev{
The offline portion of the algorithm has complexity of the order
\begin{align*}
  \overbrace{QM + Q_f^2 \cN}^{\textrm{offline preparation}} +  \overbrace{\left(N^4 + Q_L^2 N^3 + Q_f Q_L N^2 + Q_f^2\right) Q_N^{\rm trim}}^{\textrm{Evaluation of $\Delta^w_k$ over $\Xi$}} + \overbrace{N \cS(\cN)}^{\textrm{truth solver}} + \overbrace{(Q_L^2 N^2 + Q_f Q_L) \cN\,^2}^{\textrm{online preparation}} 
\end{align*}
Below are more detailed explanations of this operation count:
\begin{itemize}[itemsep=0pt]
  \item \textit{Offline preparation} --- evaluation of $C_{Q,M}$ and computations enabling optimized evaluations of $\left\|f(\bmu)\right\|_{\mathcal{X}^\cN}$.
  \item \textit{Evaluation of $\Delta^w_k$} --- $n(\Xi,k)$ evaluations for $k=1, \ldots, N$ results in the $Q^{\rm trim}_N$ factor. Each evaluation at stage $k$ requires $k^2 Q_L + k Q_f + k^3$ operations to evaluate the $k$th RBM surrogate and its PDE residual, plus $Q_f^2 + Q_L^2 k^2 + Q_f Q_L k$ operations to evaluate the error estimate.
  \item \textit{truth solver} --- $N$ solves of the truth-discretized PDE \eqref{eq:truth-pde}
  \item \textit{online preparation} --- for each $k = 1, \ldots, N$, optimizations for fast evaluations of $\Delta^w_{k+1}$ require $Q_L^2 k \cN^2 + Q_f Q_L \cN^2$, and optimizations for efficient computation of the $k$th RB surrogate require $Q_L k \cN^2 + Q_f \cN$ operations. The sum of these is domimated by the term shown.
\end{itemize}

The online portion of the algorithm has complexity on the order of $\left(N^2 Q_L + N Q_f + N^3\right)Q$. We emphasize that this complexity is $\cN$-independent, but it does depend on $Q$.


\begin{remark} From this analysis, we see that it is worthwhile to apply the hybrid algorithm only when $Q \cS(\cN)$ (the brute-force cost of \eqref{eq:uhat-truth-approximation}) is more than the total offline and online time combined. This is typically achievable when the $N$-width is small (so that $N \ll \cN$) and when $Q_f Q_L \ll \cN$. On the other hand, we note that there are extreme scenarios (such as when $Q_L$, $Q_f$, and $N$ are all large) when it is more efficient to compute gPC coefficients directly via \eqref{eq:uhat-truth-approximation}.
\end{remark}
The hybrid algorithm cost still scales linearly with $Q$, with $Q \sim P^K$ for tensorial quadrature rules. For this reason, the algorithm only delays the curse of dimensionality. However, when the RBM procedure is effective, this algorithm is still orders-of-magnitude more efficient than a direct gPC approximation via stochastic pseudospectral approximation.
}

\section{Analysis of the hybrid algorithm}
\label{sec:analysis}
In this section, we show the convergence of this goal-oriented gPC-RBM algorithm. More precisely, we show that the error committed by the hybrid RBM algorithm is controlled by the user-defined input parameter $\varepsilon_{\mathrm{tol}}$ in Algorithm \ref{euclid}. Given a {$P$-th} order $M$-term gPC projection \eqref{eq:eq1}, the $m$-th truth Fourier coefficient $\widehat{u}^{\mathcal{N}}_m$ is evaluated by \eqref{eq:uhat-truth-approximation} and its surrogate  $\widehat{u}^{N}_m$ by \eqref{eq:uhat-truth-surrogate}.  The truth and reduced basis stochastic solutions are then, respectively,
\begin{align}\label{eq:umn-truth}
  u^{\mathcal{N}}_M(\boldsymbol{x},\boldsymbol{\mu}) &= \sum_{m = 1}^{M} \hat{u}^{\mathcal{N}}_{m}(\boldsymbol{x})\Phi_{m}(\boldsymbol{\mu}), \\\label{eq:umn-rb}
  {u}^N_M(\boldsymbol{x},\boldsymbol{\mu}) &= \sum_{m = 1}^{M} \hat{u}^{N}_{m}(\boldsymbol{x})\Phi_{m}(\boldsymbol{\mu}). 
\end{align}
The properties of the RBM algorithm allow us to bound the error between the computationally expensive $\cF\left[ u^{\mathcal{N}}_M \right]$ and the efficiently computable hybrid surrogate $\cF \left[ u^N_M \right]$. This bound is our main theoretical result and is shown in Theorem \ref{thm:qoi-convergence}.

To begin, we first need to control the error between the full stochastic quadrature gPC coefficients $u^{\cN}_m\left(\bx\right)$ and its RBM surrogate $u^N_m\left(\bx\right)$.

\begin{lemma}
\label{lemma:gpc-coe-bound}
Given a $Q$-point quadrature rule $\left\{\bmu^q, w_q\right\}_{q=1}^Q$ and an $N$-dimensional reduced basis approximation $u^N(\bx, \bmu)$ for the solution $u(\bx,\bmu)$ to the PDE \eqref{eq:pde}, the error in the $m$-th gPC coefficient is given by 
  \begin{align*}
    \left\| \widehat{u}\truth_m(\bx) - \widehat{u}^N_m(\bx) \right\|_{\yc{X_{\mathcal N}}} \leq B_{Q,m} \sqrt{\frac{1}{Q} \sum_{q=1}^Q (\Delta^w_N)^2\left(\bmu^q\right)},
  \end{align*}
  where $B_{Q,m}$ is the uncentered second moment of $\Phi_m(\bmu)$ under the discrete measure defined by the quadrature rule:
  \begin{align} \label{eq:BQm-definition}
    B_{Q,m} &= \sqrt{\sum_{q=1}^Q |w_q| \Phi_m^2(\bmu^q)}.
  \end{align}
\end{lemma}

\begin{proof}
  We use the quadrature representation for these functions to prove the result:
  \begin{align*}
    \left\| \widehat{u}\truth_m(\bx) - \widehat{u}^N_m(\bx)\right\|_{\yc{X_{\mathcal N}}} &= \left\| \sum_{q=1}^Q w_q \Phi_m\left(\bmu^q\right) \left[ u\truth\left(\bx, \bmu^q\right) - u^N\left(\bx,\bmu^q\right) \right]\right\|_{\yc{X_{\mathcal N}}} \\
                                                                                &\leq \sum_{q=1}^Q \left| \sqrt{|w_q|} \Phi_m\left(\bmu^q\right) \right| \left(\sqrt{|w_q|} \left\|e(\bx,\bmu^q)\right\|_{\yc{X_{\mathcal N}}}\right) \\
                                                                                &\leq \sqrt{\sum_{q=1}^Q |w_q| \Phi_m^2\left(\bmu^q\right) } \sqrt{\sum_{q=1}^Q |w_q| \left\| e(\bx,\bmu^q) \right\|^2_{\yc{X_{\mathcal N}}}} \\
                                                                                &\leq B_{Q,m} \sqrt{\sum_{q=1}^Q \frac{1}{Q} \yc{(\Delta^w_N)^2}\left(\bmu^q\right)}
  \end{align*}
\end{proof}

This result bounds the error in each gPC coefficient as a product of two terms: the first term $B_{Q,m}$ is a measure of the (polynomial degree of) accuracy of the quadrature rule, which is computable and independent of the solution $u$. The second term is an average of the weighted \textit{a posteriori} error estimate over parameter space. 

We can now bound the RBM error in the quantity of interest.
\begin{theorem}\label{thm:qoi-convergence}
Given an $M$-term gPC projection \eqref{eq:eq1} and an $N$-dimensional reduced basis approximation \eqref{eq:rbm-projection},
  the error in the quantity of interest computed from the RBM-gPC approximation $u^N_M$, and that computed from the truth gPC approximation $u\truth_M$ is 
  \begin{align}\label{eq:qoi-convergence}
    \left\| \mathcal{F} \left[u^N_M\right] - \mathcal{F}\left[u\truth_M\right] \right\|_{\yc{X_{\mathcal N}}} \leq C_{\rm Lip} \, C_{Q,M} \sqrt{\yc{\frac{1}{Q}}\sum_{q=1}^Q \left[ \Delta^{w}_N(\bmu^q)\right]^2 },
  \end{align}
  where $C_{\mathrm{Lip}}$ is the Lipschitz constant of $\cF$ defined in \eqref{eq:F-lipschitz}, and $C_{Q,M}$ is a constant independent of $u$, defined by
  \begin{align}
    C_{Q,M} &= \sum_{m=1}^M B_{Q,m} \left|\mathcal{F} \left[ \Phi_m(\bmu) \right]\right|, 
  \end{align}
  with $B_{Q,m}$ defined in \eqref{eq:BQm-definition}.
\end{theorem}

\begin{proof}
  We begin by using our assumption \eqref{eq:F-affine} regarding the affine dependence of $\mathcal{F}$ on a gPC representation:
  \begin{align*}
    \mathcal{F} \left[u^N_M\right] - \mathcal{F}\left[u\truth_M\right] &= \sum_{m=1}^M \left(\theta_{\mathcal{F}}(\widehat{u}^N_m(\bx)) - \theta_{\mathcal{F}}(\widehat{u}\truth_m(\bx))\right) \mathcal{F}\left[\Phi_m(\bmu)\right].
  \end{align*}
Applying triangle inequality and Lipschitz continuity of $\theta_{\mathcal{F}}(\cdot)$, we have 
  \begin{align*}
    \left\| \mathcal{F} \left[u^N_M\right] - \mathcal{F}\left[u\truth_M\right] \right\|_{\yc{X_{\mathcal N}}} &\leq C_{\rm Lip} \, \sum_{m=1}^M \left\| \widehat{u}\truth_m(\bx) - \widehat{u}^N_m(\bx)\right\|_{X_{\mathcal{N}}} \left|\mathcal{F}\left[ \Phi_m(\bmu)\right]\right|.
  \end{align*}
  Using Lemma \ref{lemma:gpc-coe-bound} on the right-hand side, we obtain 
  \begin{align*}
    \left\| \mathcal{F} \left[u^N_M\right] - \mathcal{F}\left[u\truth_M\right] \right\|_{\yc{X_{\mathcal N}}} &\leq C_{\rm Lip} \, \sum_{m=1}^M B_{Q,m} \left|\mathcal{F}\left[ \Phi_m(\bmu)\right] \right| \sqrt{\yc{\frac{1}{Q}}\sum_{q=1}^Q (\Delta^w_N)^2(\bmu^q)},
  \end{align*}
  and this proves \eqref{eq:qoi-convergence}.
\end{proof}

\begin{corollary}\label{corollary:qoi-convergence-bound}
  The output gPC approximation $u^N_M$ from Algorithm \ref{euclid} satisfies
  \begin{align*}
    \left\| \mathcal{F} \left[u^N_M\right] - \mathcal{F}\left[u\truth_M\right] \right\|_{X_{\mathcal{N}}} \leq C_{\rm Lip} \, \varepsilon
  \end{align*}
  where $\varepsilon$ is defined in \eqref{eq:epsilon-def}. In particular, $\varepsilon$ is dominated by the user-prescribed $\epsilon_{\mathrm{tol}}$ if Algorithm \ref{euclid} terminates successfully.
\end{corollary}

\begin{remark} 
  The Lipschitz constant $C_{\mathrm{Lip}}$ is trivially $1$ when ${\mathcal F}[\cdot]$ is the expected value operator. For the other two cases listed in \eqref{eq:variance} and \eqref{eq:rho-norm}, it is finite as long as we have uniform stability with respect to the parameter $\bmu$ for the computational solver \eqref{eq:truth-pde}. As the discussion around \eqref{eq:uniform-boundedness} indicates, this is a standard assumption, and in that case $C_{\mathrm{Lip}} = 2 U$.
\end{remark}

\begin{remark}
We emphasize that $C_{Q,M}$ is a scalar whose value is \textit{independent} of the solution $u$, its truth discretization $u\truth$, or the RBM surrogate $u^N$. It depends only on the choice of quadrature rule, the gPC order, and the choice of what quantity of interest is to be computed. 
\end{remark}

The coefficient $C_{Q,M}$, depending on $B_{Q,m}$, is not analytically computable in general since it depends on the chosen quadrature rule in parameter space. However, the next lemma shows that \rev{for one of the two quadrature rules we are using in this paper, the tensor-product Gauss quadrature rule, we have $B_{Q,m} \equiv 1$ if the accuracy is sufficiently high.}
\begin{lemma}\label{lemma:BQm-gauss-quadrature}
Let $\left\{\Phi_1, \ldots, \Phi_M \right\}$ span the degree-$P$ isotropic total-degree space, so that $M = \left(\begin{array}{c} K+P \\P \end{array}\right)$. Assume that the quadrature rule \eqref{eq:dense-quadrature-rule} corresponds to an isotropic tensor product quadrature rule with $q$ points in each dimension, totaling $Q = q^K$ nodes. If $P < q$, then $B_{Q,m} \equiv 1$ for $m = 1 \ldots, M$.
\end{lemma}

\begin{proof}
  This is a simple consequence of the fact that a $q$-point Gaussian quadrature rule exactly integrates polynomials up to degree $2 q -1$. With $\alpha$ the size-$K$ multi-index corresponding to the linear index $m$, then 
  \begin{align*}
    \Phi_m^2(\bmu) &= \prod_{k=1}^K \phi^2_{\alpha_k}(\mu_k), & \alpha_k \leq |\alpha| = P.
  \end{align*}
  Thus, in dimension $k$, we have $\deg \phi^2_{\alpha_k} \leq 2 P < 2 q$. Additionally, a Gauss quadrature rule has all positive weights. Therefore, the quadrature rule integrates the polynomial in each dimension exactly, so $B^2_{Q,m} = \sum_{j=1}^Q w_j \Phi_m^2\left(\bmu^j\right) = \E \Phi^2_m(\bmu) = 1$.
\end{proof}

Of course, similar statements about $B_{Q,m}$ can be made for non-total-degree or anisotropic spaces so long as one has a good understanding of the quadrature rule. Even if one cannot analytically derive values for $B_{Q,m}$, it is easily and inexpensively computable by applying the quadrature rule to the gPC basis $\Phi^2_m$. 

\begin{table}[htpb]
 \caption{Summary of explicit values for bounding constants $B_{Q,m}$ and $C_{Q,M}$ for common quantity of interest operators $\cF$. In the table, $\delta_{j,k}$ is the Kronecker delta function.}\label{tab:qoi-table}
  \begin{center}
  \resizebox{\textwidth}{!}{
    \renewcommand{\tabcolsep}{0.4cm}
    \renewcommand{\arraystretch}{1.5}
    {\small
      \begin{tabular}{@{}cccp{4cm}@{}}\toprule
      Quantity of interest & $B_{Q,m}$ & $C_{Q,M}$ & Assumptions \\ \midrule
        Mean value, $\mathcal{F} = \E$ & 1 & 1 & The quadrature rule \eqref{eq:dense-quadrature-rule} exactly integrates the constant function, all weights are positive\\
        Variance, $\mathcal{F} = \mathrm{var}$ & $\sqrt{\sum_{q=1}^Q |w_q| \bPhi^2_m(\bmu_q)}$ & $\sum_{m=2}^M \sqrt{\sum_{q=1}^Q |w_q| \bPhi^2_m(\bmu_q)}$ & \yanlai{None} \\
                                               & $1 - \delta_{m,1}$ & $M-1$ & Assumptions of Lemma \ref{lemma:BQm-gauss-quadrature}\\
        $L^2_\rho$-norm squared, $\cF\left[\cdot\right] = \left\| \cdot\right\|^2_{L^2_\rho}$ & $\sqrt{\sum_{q=1}^Q |w_q| \bPhi^2_m(\bmu_q)}$ & $\sum_{m=1}^M \sqrt{\sum_{q=1}^Q |w_q| \bPhi^2_m(\bmu_q)}$ & \yanlai{None} \\
                                                                                              & $1$ & $M$ & Assumptions of Lemma \ref{lemma:BQm-gauss-quadrature}\\
    \bottomrule
    \end{tabular}
  }
    \renewcommand{\arraystretch}{1}
    \renewcommand{\tabcolsep}{12pt}
  }
  \end{center}
\end{table}

Thus under certain assumptions, the constants $C_{Q,M}$ and $B_{Q,m}$ are explicitly computable. We summarize some of these results in Table \ref{tab:qoi-table}. \rev{In the case of a sparse grid quadrature rule, we cannot compute these constants analytically, but they can be computed numerically with ease using available sparse grid software (e.g., \cite{Tasmanian_Code}).}

As an illustration, consider dimension $K=4$ with two different tensorial gPC basis sets: Legendre polynomials, and Jacobi polynomials (with parameters $\alpha = 1$, $\beta=1$). We use a Gauss-Patterson-based sparse grid and compute $B_{Q,m}$. The results are shown in Figure \ref{fig:bound}. It is difficult to discern a pattern for the constants shown in the figure, but they are all $\mathcal{O}(1)$ for the range of parameters shown.
\begin{figure}[h]
\centering
\includegraphics[width =0.49 \textwidth]{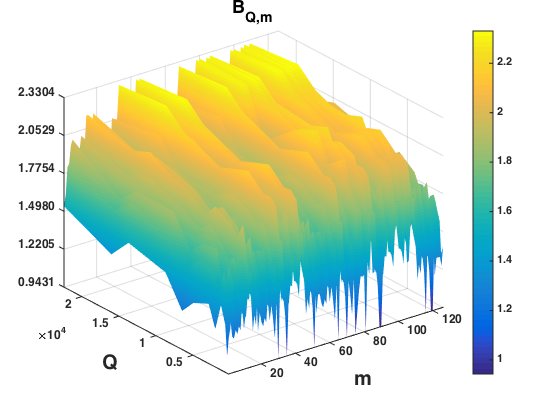}
\includegraphics[width = 0.49 \textwidth]{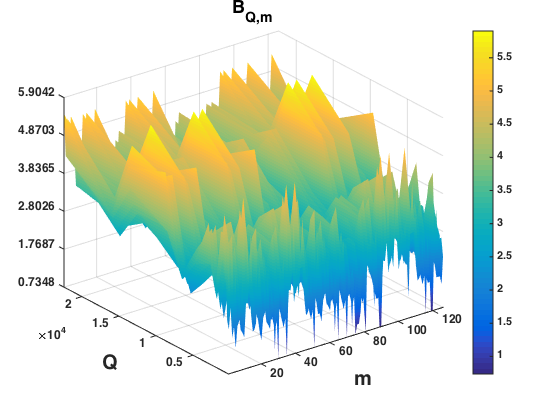}
\caption{Value of $B_{Q,m}$ with Legendre polynomial (left) and Jacobi polynomial (right). $Q$ varies with the accuracy level of sparse grid (level 7 to level 16). m records the sequence of gPC bases with $degree \le 5$   }
\label{fig:bound}
\end{figure}

\section{Numerical results}
\label{sec:results}

In this section, we present numerical results to illustrate the accuracy of the proposed hybrid approach 
and its efficiency compared to the conventional gPC method. 
The PDE with random inputs is the following  linear elliptic equation posed on the spatial domain $D = [-1,1] \times [-1,1]$ with homogeneous Dirichlet boundary conditions, 
\begin{align}
\label{eq:numerical-pde}
\begin{cases}
- \nabla \cdot (a(\boldsymbol{x} , \boldsymbol{\mu}) \nabla u(\boldsymbol{x},\boldsymbol{\mu}) ) = f  \quad   \mbox{ in }\quad  D \times \boldsymbol{\Gamma},\\
u(\boldsymbol{x},\boldsymbol{\mu}) = 0  \quad  \quad \quad \quad \quad \quad \qquad  \mbox{ on } \quad \partial D \times  \boldsymbol{\Gamma}.
\end{cases}
\end{align}
The diffusion coefficient $a(\boldsymbol{x},\boldsymbol{\mu}) $ is defined as:
\begin{equation*}
a(\boldsymbol{x},\boldsymbol{\mu}) = A + \sum_{k=1}^{K} 
\frac{{\cos(30 \mu_{k}-1)}}{k^2}\cos(kx)\sin(ky), 
\end{equation*}
where $K$ is the parameter dimension and \rev{ we set $A$ to be a positive constant 
that is large enough so that the equation is uniformly elliptic on $D$.}  We take as the right hand side $f = 1$. 
\rev{As a first step, we test our algorithm on this canonical example \cite{Evans1998} which is often tested in the 
gPC community \cite{xiu_high-order_2005, xiu_wiener--askey_2002}.} As we mentioned in \eqref{eq:F-affine}, the output of interest $\mathcal{F}$ is defined as a certain functional of the solution over the physical domain $D$. 
Here, we explore the following two cases:
\begin{itemize}
\item $\mathcal{F} = \E$. By \eqref{eq:mean-value}, the error of the mean value we observe is
\begin{align}
\label{eq:error-mean}
\xi_{\rm mean} = \left\lVert \mathcal{F}[u_M^{\mathcal{N}}] - \mathcal{F}[u_M^{N}] \right\rVert_{\ell^2(D)} = \left\lVert\widehat{u}^{\mathcal{N}}_{1}(\boldsymbol{x}) - \widehat{u}^{N}_{1}(\boldsymbol{x})\right\rVert_{\ell^2(D)}
\end{align}
\item $\mathcal{F} = \lVert \cdot \rVert^2_{L_{\rho}^2}$. By \eqref{eq:rho-norm},  the error of norm-squared operator 
we evaluate is
\begin{align}
\label{eq:error-norm}
\xi_{\rm norm} = \left\lVert \mathcal{F}[u_M^{\mathcal{N}}] - \mathcal{F}[u_M^{N}] \right\rVert_{\ell^2(D)}  = \left\lVert\sum_{m=1}^{M}(\widehat{u}_{m}^{\mathcal{N}}(\boldsymbol{x}))^2 - \sum_{m = 1}^{M}(\widehat{u}_{m}^{N}(\boldsymbol{x}))^2\right\rVert_{\ell^2(D)}
\end{align}
\end{itemize}  

In our numerical experiments, we test  problem \eqref{eq:numerical-pde} with $A = 5$ and $K = 2,4,6$.  
For the gPC approximation, we use the degree-$5$ total degree space, $\mathcal{U}_K^5$ defined in \eqref{eq:total-degree-definition}.  
We adopt a tensor product quadrature rule for lower dimensional case $K = 2$, 
and a Gauss-Patterson-based sparse grid for the higher dimensional cases $K = 4,6$. See Table \ref{tab:num-quad} 
for the nodal count $Q$ and dimension $M$ of the gPC approximation space for each $K$.  
The truth approximation 
is the solution from a pseudospectral solver 
on a $\mathcal{N} = 35 \times 35$ spatial grid.
The online solver of the goal-oriented reduced basis method (i.e., the projection operator $\mathcal{P}$ defined in \eqref{eq:rbm-projection}) is the least squares reduced collocation method developed in \cite{Yc}. 
We test the algorithm for two different probability distributions for the random variable $\bmu$, namely, the uniform distribution and Beta distribution with shape parameters $\alpha  = 2, \beta = 2$. \\

We plot the expected value $\E(u_M^\mathcal{N})$ in Figure \ref{fig:ExpValue}. The error estimates $ \varepsilon$ defined in \eqref{eq:epsilon-def} with $\mathcal{F} = \E $ and $\mathcal{F} = \lVert \cdot \rVert^2_{L_{\rho}^2}$ are plotted in Figure \ref{fig:error-estimate-mean} and Figure \ref{fig:error-estimate-norm} respectively, 
both displaying exponential convergence.
We then evaluate the actual error in the quantity of interest of the resulting surrogate solution, as defined in \eqref{eq:error-mean} and \eqref{eq:error-norm}. These are 
shown in Figure \ref{fig:error} for the two choices, $\E(\cdot)$ and $\lVert \cdot \rVert^2_{L^2_\rho}$. 
We note that both $\xi_{\textrm{mean}}$ and $\xi_{\textrm{norm}}$  have a clear exponential trend in convergence 
for both probability distributions. 
Finally, we measure the efficiency of the hybrid algorithm by calculating the ratio of the runtime between those 
of the hybrid and the traditional gPC quadrature approach. This is plotted in Figure \ref{fig:efficiency}.  Here the time for the hybrid algorithm includes \textit{both} offline and online time. 
It is clear that the proposed hybrid gPC-RBM method can reach a high level of accuracy (Figure \ref{fig:error}) 
while significantly alleviating the computational burden (Figure \ref{fig:efficiency}).
Moreover, we observe that the efficiency is increasing as $K$ gets larger, \rev{at least for the type of equations that is currently tested}. Therefore, this alleviation is more 
significant for high-dimensional problems, indicating great potential of the hybrid approach for larger parametric dimensions.

\begin{table}[htbp]
\caption{The number of quadrature nodes $Q$ and gPC approximation space dimension $M$ as a function of parametric dimension $K$.}
\label{tab:num-quad}
\begin{center}
\renewcommand{\arraystretch}{1.3}
{\begin{tabular}{|c|c|c|c|} 
\hline 
K & 2  & 4   & 6  \\
\hline
Q(K) & 1,600 & 22,401 & 367,041 \\
\hline
M(K) & 21 & 126 & 462 \\
\hline
\end{tabular}}
\end{center}
\end{table}

\begin{figure}[h]
\centering
\includegraphics[width = 0.49\textwidth]{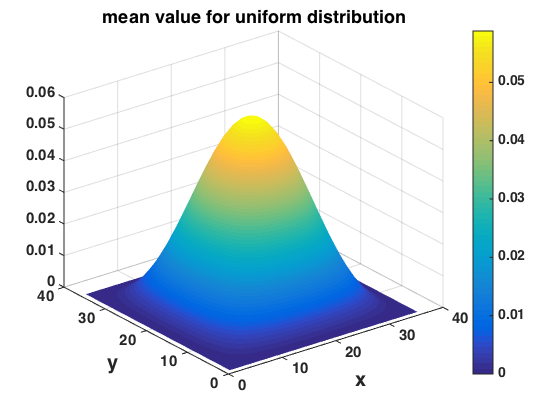}
\includegraphics[width = 0.49\textwidth]{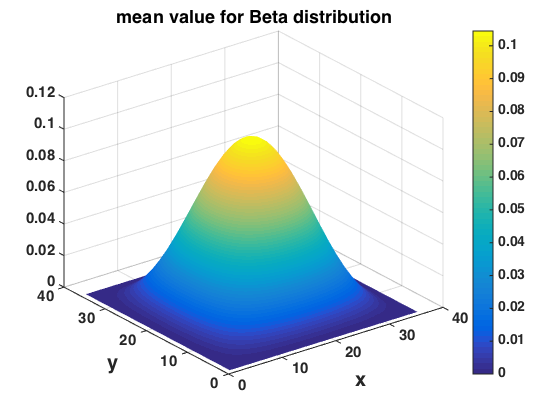}
\caption{Expected value when $K = 4$  for uniform (left) and Beta (right) distribution.}
\label{fig:ExpValue}
\end{figure}

\begin{figure}[h]
\centering
\includegraphics[width = 0.49 \textwidth]{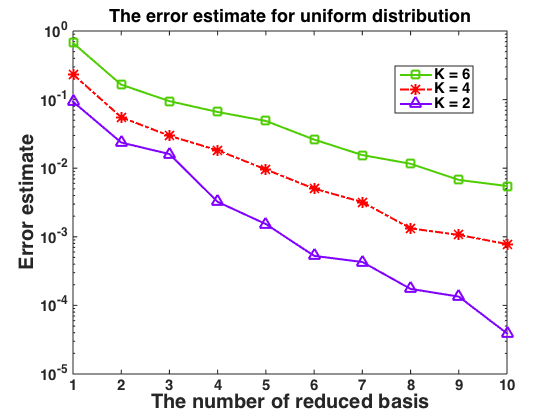}
\includegraphics[width = 0.49 \textwidth]{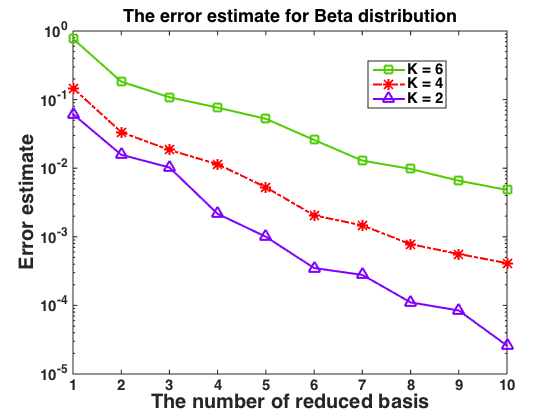}
\caption{Error estimate $ \varepsilon$ with $\mathcal{F} = \E$ for uniform distribution (left) and Beta distribution (right).}
\label{fig:error-estimate-mean}
\end{figure}

\begin{figure}[h]
\centering
\includegraphics[width = 0.49 \textwidth]{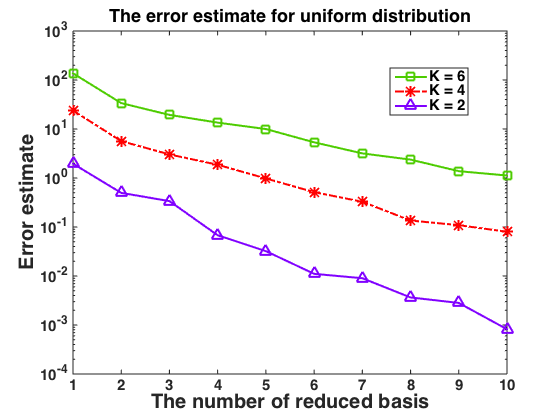}
\includegraphics[width = 0.49 \textwidth]{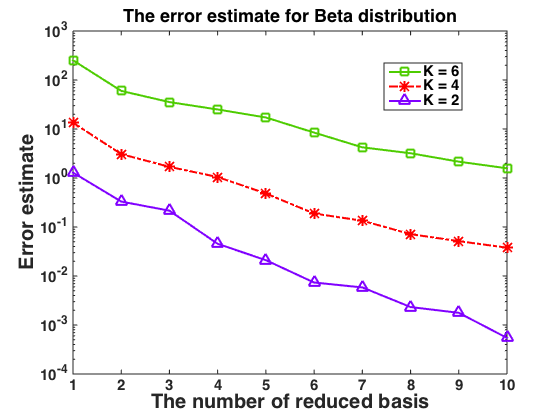}
\caption{Error estimate $ \varepsilon$ with $\mathcal{F} = \lVert \cdot \rVert^2_{L_{\rho}^2}$ for uniform distribution (left) and Beta distribution (right).}
\label{fig:error-estimate-norm}
\end{figure}

\begin{figure}[h]
\centering
\includegraphics[width = 0.49 \textwidth]{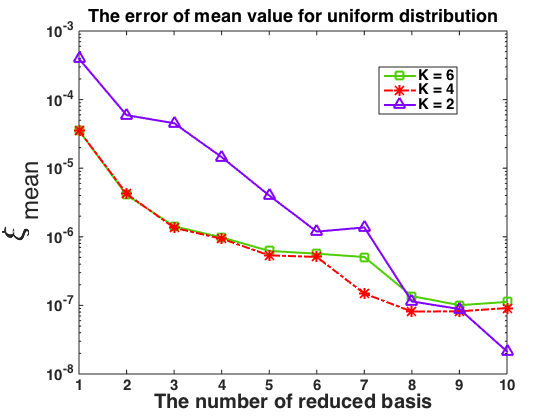}
\includegraphics[width = 0.49 \textwidth]{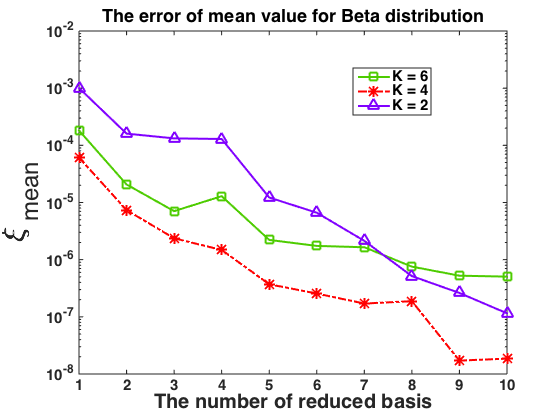}
\includegraphics[width = 0.49 \textwidth]{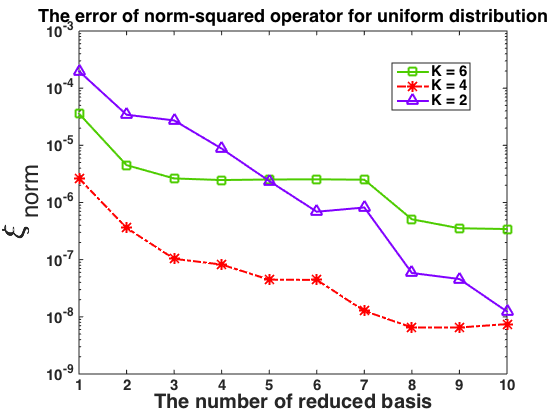}
\includegraphics[width = 0.49 \textwidth]{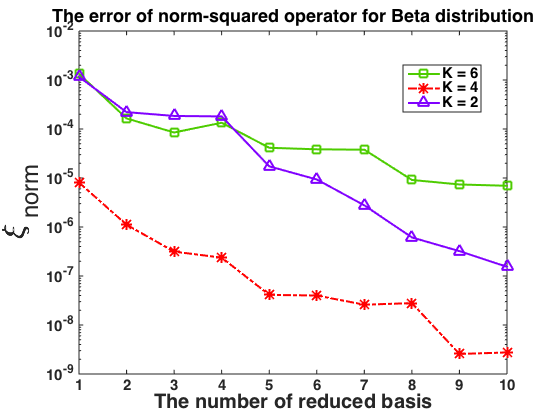}
\caption{Error of the expected value (top) and norm-squared operator (bottom) for uniform distribution (left) and Beta distribution (right).}
\label{fig:error}
\end{figure}

\begin{figure}[h]
\centering
\includegraphics[width = 0.49 \textwidth]{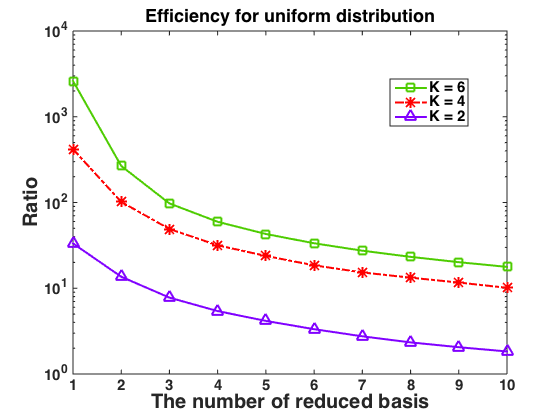}
\includegraphics[width = 0.49 \textwidth]{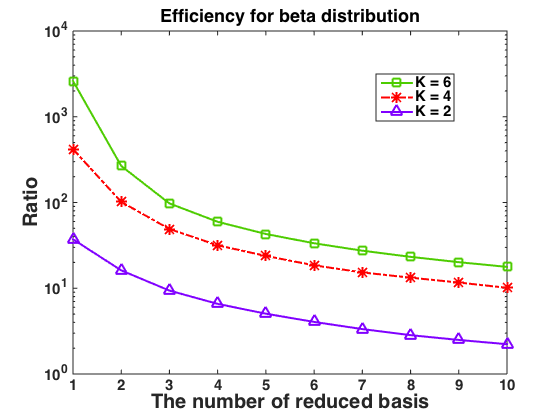}
\caption{Efficiency of the hybrid algorithm for uniform distribution (left) and Beta distribution (right).}
\label{fig:efficiency}
\end{figure}

\section{Conclusion}
We propose, analyze, and numerically test a hybridized RBM-gPC algorithm. It is based on a newly designed 
weighted RBM enabling a particular greedy algorithm tailored for any applicable quantity of interest in the context 
of uncertainty quantification. The final algorithm is analyzed to be reliable, and tested to be accurate. 
We observe that the efficiency of the hybrid algorithm increases with respect to the dimension of parameter in the partial 
differential equation. This suggests that the hybrid approach may be useful in alleviating the curse of dimensionality for other problem as well.

\providecommand{\bysame}{\leavevmode\hbox to3em{\hrulefill}\thinspace}
\providecommand{\MR}{\relax\ifhmode\unskip\space\fi MR }
\providecommand{\MRhref}[2]{%
  \href{http://www.ams.org/mathscinet-getitem?mr=#1}{#2}
}
\providecommand{\href}[2]{#2}

\bibliographystyle{siamplain}
\bibliography{hybrid-gpc}

\end{document}